\documentclass[11 pt]{amsart}
\usepackage{scrtime}
\usepackage{graphicx}
\usepackage{amssymb}
\usepackage{amsmath}
\usepackage{amsthm}
\usepackage{amscd}
\usepackage[all,2cell,dvips]{xy}
\UseAllTwocells \SilentMatrices
\newtheorem{thm}{Theorem}[section]

\newtheorem{lem}[thm]{Lemma}
\newtheorem{exm}[thm]{Example}

\newtheorem{prop}[thm]{Proposition}
\theoremstyle{definition}
\newtheorem{defn}[thm]{Definition}
\theoremstyle{remark}
\newtheorem{rem}[thm]{\bf Remark}

\numberwithin{equation}{section}

\newcommand{\QB}{\operatorname{Q}}
\newcommand{\PB}{\operatorname{P}}
\newcommand{\Hom}{\operatorname{Hom}}
\newcommand{\Mod}{\operatorname{Mod}}
\newcommand{\Id}{\operatorname{Id}}

\newcommand\sbullet{{\scriptscriptstyle \bullet}}
\newcommand\scirc{{\scriptscriptstyle \circ}}
\newcommand\cdotl{\cdot_{\scriptscriptstyle L}}
\newcommand\cdotr{\cdot_{\scriptscriptstyle R}}

\makeatletter
\newcommand{\rmnum}[1]{\romannumeral #1}
\newcommand{\Rmnum}[1]{\expandafter\@slowromancap\romannumeral #1@}
\makeatother
\SelectTips{cm}{}
\calclayout

\begin{document}
\title[ (Quasi-)Poisson Enveloping Algebras]{(Quasi-)Poisson enveloping algebras}

\author{Yan-Hong Yang}
\address{Department of Mathematics, Columbia University, New York, New York 10027, USA;}
\email{yhyang$\symbol{64}$math.columbia.edu }

\author{Yuan Yao}
\address{Department of Mathematics, The University of Texas, 1 University Station C1200,
Austin, TX 78712, USA;}
\email{yyao$\symbol{64}$math.utexas.edu}

\author{Yu Ye$^{*}$}
\address{Yu Ye\\ Department of Mathematics\\ University of Science and
Technology of China\\ Hefei 230026, Anhui\\ PR China.}
\email{yeyu@ustc.edu.cn}

\thanks{$^*$ The corresponding author}

\maketitle

\begin{abstract}
 We introduce the quasi-Poisson enveloping algebra and Poisson enveloping
algebra for a non-commutative Poisson algebra. We prove that
for a non-commutative Poisson algebra, the category of quasi-Poisson modules is equivalent to
the category of left modules over its quasi-Poisson enveloping algebra, and the category of Poisson
modules is equivalent to the category of left modules over its
Poisson enveloping algebra.
\end{abstract}

\setcounter{section}{-1}
\section{Introduction}

Poisson algebras appear naturally in Hamiltonian mechanics, and play
a central role in the study of Poisson geometry and quantum groups. Poisson
manifolds are known as smooth manifolds equipped with a Poisson algebra structure.
One can consult \cite{va} for topics in Poisson geometry.

With the development of non-commutative geometry in the past
decades, versions of non-commutative Poisson-like structures
were introduced and studied by many authors from different perspectives
\cite{ko,lo,xu,fgv,rvw,ceg,vdb,fl,cp,yyz}.
Non-commutative Poisson algebra (NCPA) is one of them. The notion of NCPA is first
suggested by P. Xu \cite{xu}, which is especially suitable for geometric situation, and
developed by Flato, Gerstenhaber and Voronov to a more general case \cite{fgv}.

An important and practical idea to understand an algebraic object is to study the representations, or equivalently,
the modules over it. So is for Poisson algebra. Poisson modules over an NCPA are defined
in a natural way \cite{fgv}, see also in \cite[Section 0]{ku1}.
Note that when restricted to  a usual Poisson algebra, this notion of Poisson
module is different from the usual one studied in \cite[Definition 1]{oh} and \cite{fa}; see Remark \ref{poi-difference}.

In this paper, we introduce quasi-Poisson enveloping algebra and Poisson enveloping algebra
for a given NCPA. It turns out that the category of quasi-Poisson modules is
equivalent to the category of left modules over its quasi-Poisson enveloping algebra, and the
 category of Poisson modules is equivalent to the category of left modules over its Poisson enveloping
 algebra.

As a consequence, we give an affirmative answer to a question raised in
\cite[Section 2]{fgv}: does there exist enough injectives or projectives in the category of Poisson
modules? This will enable us to construct certain cohomology theory for NCPA in a standard way by using
projective and injective resolutions \cite{bh}. A possible application for algebraist is that cohomology theory of the standard Poisson algebras (see Example \ref{standard}) may give some brand new invariants of associative algebras.

Note that another version of Poisson enveloping algebra for a commutative Poisson algebra
was introduced in \cite[Definition 3]{oh}. We would like to emphasize that
Poisson enveloping algebra in our sense, when
restricted to a commutative Poisson algebra, is different from that one; see Remark \ref{Poi-env-difference}.

The paper is organized as follows. In section 1, we recall some definitions and notations. A characterization of Poisson-simple algebras is also given. Section 2 deals with the construction of quasi-Poisson enveloping algebra and
Poisson enveloping algebra for a given NCPA. In Section 3, we give our main result: the category of quasi-Poisson modules is equivalent to the category of left modules over the quasi-Poisson enveloping algebra; the category of Poisson modules is equivalent to the category of left modules over the Poisson enveloping algebra.

\section{Poisson algebras and modules}

Throughout $\mathbb{K}$ will be a fixed field of characteristic 0 and algebras considered are over $\mathbb{K}$. All
unadorned $\otimes$ will mean $\otimes_{\mathbb{K}}$. Without loss of generality, we assume that associative algebras
considered here have an identity.

\begin{defn} A \textbf{non-commutative Poisson algebra} (NCPA) over $\mathbb{K}$ is a triple  $A = (A, \sbullet, \{-,-\})$, where $A$ is a $\mathbb{K}$-space, $(A, \sbullet)$ is an associative algebra and $(A, \{-,-\})$ is a Lie algebra, such that the Leibniz
rule $\{ab,c\}=a \{b,c\}+\{a,c\} b$ holds for all $a,b,c\in A$, here we write $a\sbullet b$ as $ab$ for brevity. The operator $\sbullet$ is called the multiplication and $\{-,-\}$ the Poisson bracket of $A$.
\end{defn}

\begin{rem}Note that a usual Poisson algebra is by definition an NCPA with commutative multiplication. In this paper, when we use the term Poisson algebra we always mean a commutative one, otherwise we use the notion NCPA.
\end{rem}

\begin{exm}\label{standard} Any associative algebra $A$ is also a Lie algebra with the Lie bracket given by the commutator $[-,-]$, say $ [a, b] = a b - b a$ for any $a, b\in A$. One shows easily that $(A, \sbullet, [-,-])$ is an NCPA, which we called the standard NCPA of $A$.
\end{exm}

\begin{defn}
Let $A$ be an NCPA. A \textbf{quasi-Poisson $A$-module} $M$ is a an $A$-$A$-bimodule $M$ together with a
$\mathbb{K}$-bilinear map $\{-,-\}_* \colon A\times
M\longrightarrow M,$ which satisfies
\begin{align}\label{pm1}
&\{a,b\cdot m\}_*=\{a,b\}\cdot m+b\cdot \{a,m\}_*,
\\\label{pm2}&\{a,m\cdot b\}_*=m\cdot \{a,b\}+\{a,m\}_*\cdot b,\\
\label{pm3} &\{\{a,b\},m\}_*=\{a,\{b,m\}_*\}_* -\{b,\{a,m\}_*\}_*
\end{align}
for all $a,b\in A$ and $m\in M$, and $a\cdot m$ and $m\cdot a$ denote the left and right
$A$-action on $M$ respectively. Note that the last condition is just equivalent to say that
$M$ is a Lie module over $A$. If moreover,
\begin{align}
\label{pm4}\{a  b,m\}_*=a\cdot \{b,m\}_*+\{a,m\}_*\cdot b
\end{align}
holds for all $a, b\in A$ and $m\in M$, then $M$ is called a \textbf{Poisson $A$-module}.
A \textbf{homomorphism} of quasi-Poisson $A$-modules (\emph{resp.} Poisson modules) is a $\mathbb{K}$-linear
map $f\colon M\rightarrow N$ such that
\begin{align}
\label{pm5}&f(a\cdot m)=a\cdot f(m),\ \ f(m\cdot a)=f(m)\cdot a,
\\\label{pm6}&f(\{a,m\}_*)=\{a,f(m)\}_*
\end{align} hold for all $a\in A$ and $m\in M$.
We denote the category of quasi-Poisson $A$-modules
(\emph{resp.} Poisson modules) by ${\QB(A)}$
(\emph{resp.} ${\PB(A)}$).
\end{defn}

\begin{rem} The definition of Poisson $A$-module is suggested in \cite{fgv},
 see also \cite[Section 0]{ku1}. One of its immediate merits is that
$A$ itself can be naturally viewed as a Poisson $A$-module with the regular $A$-$A$-bimodule structure and
the action $\{-,-\}_*$ given by the Poisson bracket, i.e., $\{a,b\}_*=\{a,b\}$
 for all $a,b\in A$.
\end{rem}

\begin{exm}  Let $A$ be an NCPA. Then $A\otimes A$ admits a quasi-Poisson module structure with the actions given by $a\cdot(b\otimes c) = a b\otimes c$, $ (a\otimes b) \cdot c = a\otimes bc$ and $\{a, b\otimes c \}_* = \{a, b\}\otimes c + b\otimes \{a, c\}$ for all $a, b, c\in A$. Note that this does not give to a Poisson module in general.
\end{exm}

\begin{exm} Let $A$ be an NCPA with trivial Poisson bracket, say $\{-,-\} = 0$. Then any $A$-bimodule $M$ is a Poisson module under trivial Lie action, i.e. $\{-, -\}_* = 0$. Note that this is not true if the Poisson bracket is non-trivial.
\end{exm}

\begin{rem}\label{poi-difference} The usual Poisson modules over commutative Poisson algebras
(\cite[Definition 1]{oh}) are just the
Poisson modules defined here with coincide left and right associative actions.

Concerning the relations between Hochschild cohomology and deformation theory, Poisson modules in our sense by using bimodules seems to be more suitable for deformation theory, even for commutative Poisson algebras.
\end{rem}

Recall that a \textbf{Poisson derivation}
is a linear endomorphism of $A$ which is both an associative algebra
derivation and a Lie algebra derivation. We have the
following observation, generalizing \cite[Example 7]{fa}.

\begin{prop}
Let $A$ be an NCPA and $C(A)$ the center of $A$ in the
associative sense. If a $\mathbb{K}$-bilinear form $\{-,-\}_*\colon\ A\times A\rightarrow A$
makes the regular $A$-$A$-bimodule $A$  a
Poisson $A$-module, then the induced map
$$\psi\colon A\rightarrow A,\ \psi(a)=\{a,1_{A}\}_*,\ \forall a\in A $$
is a Poisson derivation satisfying $\psi(A)\subseteq C(A)$ and
\begin{align}\label{reg-poi-mod}
\psi(a)[b,c]=0,\ \forall\ a,b,c\in A,
\end{align}
here $[b,c] = bc -cb$ denotes the commutator.
Conversely, a Poisson derivation $\psi\colon A\rightarrow A$ satisfying $\psi(A)\subseteq C(A)$ and
 (\ref{reg-poi-mod})  makes the regular bimodule $A$ a Poisson $A$-module with  $\{-,-\}_*$ given by $$\{a,b\}_*=\{a,b\}+\psi(a)b,\
\forall a,b\in A.$$ Moreover, if $\dim_{\mathbb{K}}C(A)=1$, then there is a
unique Poisson $A$-module structure $(A, \{-,-\}_*)$ on the regular bimodule $A$, which is given by
$\{a,b\}_*=\{a,b\}$.
\end{prop}
\begin{proof} Suppose $A$ is a Poisson $A$-module. Set
$\psi(a)=\{a,1_{A}\}_*$. We will show that $\psi$ is a Poisson
derivation with $\psi(A)\subseteq C(A)$. For any $a,b\in A$,
\begin{align*}
\{a,b\}_*&=\{a,b\}+b\{a,1_{A}\}_*
\\&=\{a,b\}+\{a,1_{A}\}_*b,
\end{align*}
where the equalities follow from $b= b1_A = 1_Ab$ and (\ref{pm1}) and (\ref{pm2}). Hence $\psi(a)\in C(A)$ for all $a\in A$.

Next, we have
$\psi(ab)=\{ab,1_{A}\}_*=\{a,1_{A}\}_*b+a\{b,1_{A}\}_*=\psi(a)b+a\psi(b)$ by (\ref{pm4}), which
implies that $\psi$ is an associative derivation.

Moreover, the following equalities implies that $\psi$ is
also a Lie algebra derivation.
\begin{align*}
\psi(\{a,b\})&=\{\{a,b\},1_{A}\}_*=\{a,\{b,1_{A}\}_*\}_*
-\{b,\{a,1_{A}\}_*\}_*\\
&=\{a,\{b,1_{A}\}_*\}+\{b,1_{A}\}_*\{a,1_{A}\}_*
-\{b,\{a,1_{A}\}_*\}-\{a,1_{A}\}_*\{b,1_{A}\}_*\\
&=\{a,\psi(b)\}+\{\psi(a),b\}
\end{align*}

At last, $\forall\ a,b,c\in A$ we have
\begin{align*}
&\{ab,c\}_* = \{ab,c\}+\psi(ab)c
=a\{b,c\}+\{a,c\}b+a\psi(b)c+\psi(a)bc, \\
&a\{b,c\}_*+\{a,c\}_*b= a\{b,c\}+a\psi(b)c+\{a,c\}b+\psi(a)cb.
\end{align*}
Combined with (\ref{pm4}), we get the equality (\ref{reg-poi-mod}).

The proof of the converse part is easy and left to the reader.

 Now suppose that $\dim_{\mathbb{K}}C(A)=1$. It suffices
to consider the case $\dim_{\mathbb{K}}A\geq3$. Take a basis
$\{e_1, e_2, e_3\cdots\}$ of $A$ with $e_1=1_{A}$. Thus $ C(A)=\mathbb{K}e_1$.
Suppose that $(A,\{-,-\}_*)$ is a Poisson $A$-module and $\psi$ the
induced derivation. Then $$\psi(e_ie_j)=\psi(e_i)e_j+e_i\psi(e_j),\
\forall i\neq j\geq2.$$  Note that $\psi(e_ie_j)\in \mathbb{K}e_1$,
$\psi(e_i)e_j\in \mathbb{K}e_j$ and $e_i\psi(e_j)\in \mathbb{K}e_i$. By assumption $e_1,\ e_i,\
e_j$ are linearly independent and it forces $\psi(e_i)=\psi(e_j)=0,\
\forall i,\ j\geq2$. Since $\psi(1_{A})=0$ for any Poisson derivation $\psi$,
we obtain that $\psi(a)=0,\forall a \in A$, and hence
$\{a,b\}_*=\{a,b\}+\psi(a)b=\{a,b\},\ \forall a,b\in A$.
\end{proof}

\begin{rem}
For an NCPA, it is still not known whether the
above correspondence between the set of Poisson module structures
on the regular $A$-bimodule and the set of Poisson derivations is bijective.
This is just equivalent to ask whether any Poisson derivation for
an NCPA has the property (1.7).
\end{rem}

In the rest of this section, we will characterize the simplicity of an NCPA $A$
by studying its Poisson modules. Recall that a \textbf{left} (\emph{resp.} right,
two-sided) \textbf{Poisson ideal} of $A$ is defined to be
an ideal in the Lie algebra sense and a left (\emph{resp.} right,
two-sided) ideal in the associative algebra sense. When $I$ is a
two-sided Poisson ideal of $A$, it is easy to check that $A/I$ is a
Poisson $A$-module with the weak Poisson action given by
$\{a,b+I\}_*\triangleq\{a,b\}+I$ for any $a,b\in A$. Moreover, $A/I$ is
an NCPA with the obvious bracket
$\{a+I,b+I\}\triangleq\{a,b\}+I$.

An NCPA $A$ is \textbf{left} ({\it resp.} right) \textbf{Poisson-simple}
provided that any left ({\it resp.} right) Poisson ideal of $A$ is either 0 or $A$.
$A$ is said to be \textbf{Poisson-simple} if any two-sided Poisson ideal of $A$
is either 0 or $A$ itself. Clearly if $A$ is simple as an associative algebra, then
$A$ is Poisson-simple.

 Let $M$ be a quasi-Poisson $A$-module. The \textbf{annihilator} $Ann_A(M)$
of $M$ in $A$ is defined as $Ann_A(M)\triangleq \{a\in A\mid a\cdot m= m\cdot a=0, \forall m\in M\}$.
Clearly $Ann_A(M)$ is a two-sided ideal of $A$ in the associative sense. Moreover, $Ann_A(M)$ is
also a Lie ideal and hence a Poisson ideal, just as shown in the following lemma.

 \begin{lem}\label{anni-ideal} Let $A$ be an NCPA and $M$ a quasi-Poisson $A$-module.
Then $Ann_A(M)$ is a two-sided Poisson ideal of $A$.
\end{lem}

\begin{proof}  Let
$a\in Ann_A(M)$ and $b\in A$. Then we have $(ba)m=b(am)=0$ and $m(ba)=(mb)a=0$, hence $ba\in Ann_A(M)$.
Similarly we have $ab\in Ann_A(M)$. Therefore $Ann_A(M)$ is a two-sided ideal
in the associative sense.

Since $M$ is a quasi-Poisson module by assumption, (\ref{pm1}) and (\ref{pm2}) implies that $\{b,a\}m=\{b,am\}_*-a\{b,m\}_*=0$ and
$m\{b,a\}=\{b,ma\}_*-\{b,m\}_*a=0$, thus $\{b,a\}\in
Ann_A(M)$. It concludes that $Ann(M)$ is a two-sided Poisson ideal.
\end{proof}

\begin{prop}
Let $A$ be an NCPA. The following statements are equivalent:
\begin{enumerate}
\item[(\rmnum{1})] $A$ is Poisson-simple;
\item[(\rmnum{2})] For every nonzero quasi-Poisson $A$-module $M$, $Ann_A(M)=0$;
\item[(\rmnum{3})] For every nonzero Poisson $A$-module $M$, $Ann_A(M)=0$.
\end{enumerate}
\end{prop}

\begin{proof}
(\rmnum{1})$\Longrightarrow$(\rmnum{2}). Assume that $A$ is Poisson simple and $M$ is a
quasi-Poisson $A$-module. $Ann_A(M)$ is a two-sided Poisson ideal of $A$ by Lemma \ref{anni-ideal}.
$M\ne 0$ implies that $Ann_A(M)\ne A$, for $1_A\notin Ann_A(M)$ in this case. Now the simplicity
of $A$ forces $Ann_A(M) =0 $.

(\rmnum{2})$\Longrightarrow$(\rmnum{3}) is obvious, for a Poisson module is always a quasi-Poisson module.

(\rmnum{3})$\Longrightarrow$(\rmnum{1}). Suppose that for any Poisson $A$-module $M\ne0$, $Ann_A(M)=0$. Assume that $A$ is not Poisson-simple and let
$I\ne0$ be a proper two-sided Poisson ideal, then $A/I$ is a Poisson $A$-module such that $Ann_A(A/I)=I$, which leads to a contradiction. Thus $A$ must be Poisson-simple.
\end{proof}

\section{ Poisson enveloping algebras}

The main purpose of this section is to introduce the notion of Quasi-Poisson enveloping algebra and Poisson enveloping algebra for an NCPA. For this we need some preparation.

Recall that the opposite algebra $A^{op}$ of an
associative algebra $A$ has the same underlying vector space as $A$
and the multiplication is given in a reversed order, i.e.
$a\cdot_{A^{op}}b=b\cdot_{A}a$. The associative algebra
$A\otimes A^{op}$ is called the \textbf{enveloping algebra} of $A$ and
denoted by $A^{e}$. We denote $a\cdot_A b$ by $a\sbullet b$ and $a\cdot_{A^{op}}b$
by $a\scirc b$ to simplify the notation. Note that if $A$ is an NCPA , then $A^{op}$ is also an NCPA with
the same Poisson bracket as defined on $A$.

Let $H$ be a Hopf algebra and $A$ an $H$-module algebra, i.e. $A$ is an associative algebra and a left $H$-module, such that $h(1_A) = \epsilon(h)1_A$, $h(ab) = \sum (h_1a)(h_2b)$ for all $h\in H$ and $a\in A$, here we use the Sweedler's notation $ \Delta(h) = \sum h_1\otimes h_2$. Then we may define a multiplication $\star$ on the vector space  $A\otimes H$ by setting $$(a\otimes h)\star (a'\otimes h') = \sum ah_1(a')\otimes h_2h'$$ for all $a, a'\in A$ and $h, h'\in H$. It is not hard to show that $(A\otimes H, \star )$ is an associative algebra with identity $1_A \otimes 1_H$, which we call the \textbf{smash product} of $A$ and $H$ and denote by $A \# H $. Elements in $A\# H$ are usually written as $a\# h$ rather than $a\otimes h$. We refer to \cite{sw} for more details about Hopf algebras, module algebras over a Hopf algebra and their smash product.

\begin{rem}\label{mod-alg-rem} Let $H$ be a cocommutative Hopf algebra. Recall that $H$ is cocommutative if $\Delta(h) = \sum h_1\otimes h_2 = \sum h_2\otimes h_1$ for any $h\in H$. Then for any $H$-module algebra $A$, the opposite algebra $A^{op}$ is also a $H$-module algebra  with the same $H$-action as the one on $A$. Moreover, if $A$ and $B$ are both $H$-module algebra, then so is $A\otimes B$, here the $H$-action on $A\otimes B$ is given by the usual tensor product, say $h\cdot(a\otimes b)=\sum h_1\cdot_A a\otimes h_2\cdot_B b$ for any $h\in H$, $a\in A$ and $b\in B$. In
particular, the enveloping algebra $A^e$ of an $H$-module algebra $A$ is also an $H$-module algebra.
\end{rem}

Let $A$ be a $\mathbb{K}$-vector space and $T(A)$ the tensor algebra. By definition
there is a decomposition of vector spaces $T(A)=\bigoplus_{i=0}^\infty
T^i(A)$, making $T(A)$ a positively graded algebra, here $T^i(A)=A^{\otimes i}$. Any $a\in T^i(A)$ is said to be a
homogeneous element of degree $i$, denoted by $\deg(a)=i$,
particularly $\deg(1_{T(A)})=0$. We may also
identify $A$ with a subspace of $T(A)$, say $T^1(A)$, in an obvious way.

The following notations will be handy later on:

(i)\quad Let $\{v_i\}_{i\in S}$ be a fixed $\mathbb{K}-$basis of $A$, here $S$ is an index
set. Let $\alpha=(i(1),i(2),\cdots,i(r))\in S^r$ be a sequence in $S$. $r $ is called the degree of $\alpha$ and denoted by $r = \deg(\alpha)$. We denote the element $v_{i(1)} v_{i(2)}\cdots v_{i(r)}=v_{i(1)}\otimes v_{i(2)}\otimes\cdots\otimes v_{i(r)}$  by $\overrightarrow{\alpha}$ and also denote $1_{T(A)}$ by $\overrightarrow{\emptyset}$ for consistency, where $\emptyset$ is the empty sequence. Thus $\{\overrightarrow{\alpha}\mid \alpha\in S^r, r>0\}\cup \{1_{T(A)}\}$
form a basis of $T(A)$. For given $\alpha=(i(1),\cdots,i(r))$
and $\beta=(j(1),\cdots,j(s))$, we define
$\alpha\vee\beta\triangleq(i(1),\cdots,i(r),j(1),\cdots,j(s))$,
hence
$\overrightarrow{\alpha} \overrightarrow{\beta} = \overrightarrow{\alpha\vee\beta}$.

(ii)\quad Let $\alpha=(i(1),\cdots,i(r))$ and $ X\sqcup Y$ be an
ordered bipartition of $\underline{\rm{r}}=\{1,\cdots, r\}$, here
"ordered" means that $X\sqcup Y$ and $Y\sqcup X$ give different
bipartitions, which differs from the usual one. We would like to emphasize that $X$ and $Y$ are allowed to be
empty here. Suppose $X=\{X_1, X_2,\cdots, X_{r_1}\}$ and $Y=\{Y_1,
Y_2,\cdots, Y_{r_2}\}$ with $X_1< X_2<\cdots<X_{r_1}$ and $Y_1< Y_2<\cdots< Y_{r_2}$
, where $r_1=|X|$, $r_2=|Y|$. Set
$\alpha_X=(i(X_1),i(X_2), \cdots, i(X_{r_1}))$ and $\alpha_Y=
(i(Y_1), i(Y_2),\cdots, i(Y_{r_2}))$. By definition
$\alpha=\alpha_{X}\sqcup\alpha_{Y}$ is called an \textbf{ordered
bipartition} of $\alpha$ with respect to the ordered bipartition
$\underline{\rm{r}}=X\sqcup Y$.
 Similarly, one defines \textbf{ordered $n$-partitions} $\alpha=
\alpha_1\sqcup\alpha_2\cdots\sqcup\alpha_n$ for any $n\ge2$.

\begin{rem} Note that by definition, different partitions $\underline{\rm{r}}=X\sqcup Y$ and
$\underline{\rm{r}}=X'\sqcup Y'$ give rise to different partitions $\alpha = \alpha_X\sqcup\alpha_Y$ and
$\alpha = \alpha_{X'}\sqcup\alpha_{Y'}$, even if $\alpha_X=\alpha_{X'}$ and $\alpha_Y=\alpha_{Y'}$
as sequences in $S$.

For example, let $i\in S$ and $\alpha = (i, i)$. Then there
are exactly 4 ordered bipartitions of $\alpha$:
\begin{align*} &\{1,2\}= \emptyset\sqcup\{1,2\}, & \alpha=\emptyset\sqcup(i,i);
   \\ &\{1,2\}= \{1\}\sqcup \{2\}, & \alpha=(i)\sqcup(i);
   \\ &\{1,2\}= \{2\}\sqcup \{1\} , & \alpha=(i)\sqcup(i);
   \\ &\{1,2\}= \{1,2\}\sqcup \emptyset, & \alpha=(i,i)\sqcup\emptyset.
  \end{align*}

\end{rem}

 By the expression $\sum_{\alpha=\alpha_1\sqcup\alpha_2}$, we mean to take the sum
over each possible ordered bipartition of
$\{1,\cdots,\deg(\overrightarrow{\alpha})\}$. Similarly,
$\sum_{\alpha=\alpha_1\sqcup\alpha_2\sqcup\alpha_3}$ means the sum
is taken over all possible ordered tripartitions of
$\{1,\cdots,\deg(\overrightarrow{\alpha})\}$, and so on.

One shows that $T(A)$ admits a Hopf algebra structure. The counit $\epsilon$ is given by $\epsilon(1_{T(A)}) = 1$ and $\epsilon(\overrightarrow{\alpha}) = 0$ for all $\alpha\in S^r$ with $r> 0$. The coproduct $\Delta$ is given by
 \[ \Delta(\overrightarrow{\alpha}) = \sum_{\alpha=\alpha_1\sqcup \alpha_2} \overrightarrow{\alpha_1}\otimes \overrightarrow{\alpha_2}\] for all $\alpha\in S^r$ with $ r\ge 0$.
 In particular, $\Delta(1_{T(A)}) = 1_{T(A)}\otimes 1_{T(A)}$ and $\Delta(a) = 1_{T(A)}\otimes a + a\otimes 1_{T(A)}$ for any $a\in A$. The coproduct $\Delta$ is known as the ''shuffle coproduct''. By definition $\Delta$ is cocommutative, and hence $T(A)$ is a cocommutative Hopf algebra.

Assume further that $A$ is a Lie algebra and consider the \textbf{universal enveloping algebra}
$\mathcal{U}(A)$ of $A$. By definition $\mathcal{U}(A) = T(A)/I$, where $I$ is
the two-sided ideal of $T(A)$ generated by $\{x\otimes y-y\otimes x-\{x,y\}\mid x,y\in A\}$. We denote the
canonical projection by $\pi\colon T(A)\twoheadrightarrow\mathcal{U}(A)$. Moreover, $I$ is easily shown to be a Hopf ideal of $T(A)$ and hence $\mathcal{U}(A)$ inherits a cocommutative Hopf algebra structure.

Let $\{v_i\}_{i\in S} $ be a fixed $\mathbb{K}-$basis of $A$ with $S$ an index set. We may
assume that $S$ is totally ordered. According to the Poincare-Birkhoff-Witt
theorem, the elements \[\{\overrightarrow{\alpha}\mid \alpha = (i(1), i(2),\cdots, i(r))\in S^r,\
i(1)\le i(2)\le\cdots \le i(r), \ r\ge 1\},\] along with $1_{\mathcal{U}(A)}$, form a basis
of $\mathcal{U}(A)$, here we also denote the element $\pi(\overrightarrow{\alpha})$ in $\mathcal{U}(A)$ by $\overrightarrow{\alpha}$ by abuse of notation. Clearly we may identify $A$ with $\pi(T^1(A))$, a subspace of $\mathcal{U}(A)$.

Recall that for a Lie algebra $A$, one of the advantages of
$\mathcal{U}(A)$ is that the category of Lie modules over $A$ is equivalent to the category of left modules over
$\mathcal{U}(A)$. Particularly, the Lie bracket makes $A$ a Lie module and hence a $\mathcal{U}(A)$-module with
the action given by $\overrightarrow{\alpha}\cdot a = \{v_{i(1)},\{v_{i(2)},\{,\cdots,\{v_{i(r)},
a \}\cdots\}\}\}$ for all $\alpha =({i(1)},i(2),\cdots, {i(r)})\in S^r$ and $a\in A$. We denote the action $\overrightarrow{\alpha}\cdot a$ by $\overrightarrow{\alpha}(a)$. Note that
$A^{op}$ has the same underlying space as $A$ and hence the same $\mathcal{U}(A)$-action. Since $\mathcal{U}(A)$ is
a Hopf algebra, $A^e = A\otimes A^{op}$ is a tensor product of $\mathcal{U}(A)$-modules and hence a  $\mathcal{U}(A)$-module.  Explicitly,
$\overrightarrow{\alpha} (a\otimes b) = \sum_{\alpha=\alpha_1\sqcup\alpha_2}\overrightarrow{\alpha_1} (a)\otimes \overrightarrow{\alpha_2} (b)$ for all $\alpha\in S^r$ with $r\ge0$, $a,b \in A$.

\begin{prop}\label{prop-A-is-mod-alg} Let $A = (A, \sbullet, \{-,-\})$ be an NCPA. Then $A$, $A^{op}$ and $A^e$ are all
$\mathcal{U}(A)$-module algebras with the above $\mathcal{U}(A)$-module structure.
\end{prop}

\begin{proof} We begin with the algebra $A$. Clearly $1_{A}$ is a center element
of the Lie algebra $A$, hence $x(1_A) =\epsilon(x)1_A$ for any $x\in \mathcal{U}(A)$.
Next we prove  that
\begin{equation}\label{mod-alg} x (a\sbullet b) = \sum (x_1(a))\sbullet (x_2(b))\end{equation}
for all $x\in \mathcal{U}(A)$ and $a, b\in A$.
We need only to check it for a set of generators. In fact, if $x, y\in \mathcal{U}(A)$ are both satisfying (\ref{mod-alg}), then
\begin{align*}(xy)(a\sbullet b) &= x (y ( a\sbullet b ))
= x (\sum y_1(a)\sbullet y_2(b)) \\
&= \sum (x_1(y_1(a)))\sbullet (x_2(y_2(b)))\\
&= \sum((x_1y_1)(a))\sbullet ((x_2y_2)(b)) \\&
= \sum ((xy)_1(a))\sbullet ((xy)_2(b)),
\end{align*}
which implies that $xy $ also satisfies (\ref{mod-alg}).

Note that for any $v\in A$, $\Delta(v) = 1_{\mathcal{U}(A)}\otimes v + v\otimes 1_{\mathcal{U}(A)}$.
By (\ref{pm4}) we have
\[v {(a\sbullet b)} = \{v, a\sbullet b\} = a\sbullet \{v, b\} +  \{v, a\} \sbullet b = \sum (v_1(a))\sbullet (v_2(b))\]
for all $a, b \in A$. Since $A$ generates $\mathcal{U}(A)$, it follows that that $A$ is a $\mathcal{U}(A)$-module algebra. By Remark \ref{mod-alg-rem}, $A^{op}$ and $A^e$ are both $\mathcal{U}(A)$-module algebras.
\end{proof}

With the above preparation we give our main definition as follows.

\begin{defn}
Let $A=(A,\sbullet,\{-,-\})$ be an NCPA. The  smash product $A^e\# \mathcal{U}(A)$ is called the \textbf{quasi-Poisson enveloping algebra} of $A$ and denoted by  $\mathcal{Q}(A)$.
The \textbf{Poisson enveloping algebra} of $A$, denoted by $\mathcal{P}(A)$, is defined to be the quotient algebra $\mathcal{Q}(A)/J$, here $J$ is the ideal of $\mathcal{Q}(A)$ generated by
\[\{1_{A}\otimes1_{A^{op}}\# (a\sbullet b)-a\otimes1_{A^{op}}\# b-1_{A}\otimes b\# a|\
 a,b\in A\}.\]
\end{defn}

\begin{rem} By definition, $\mathcal{Q}(A) = A\otimes A^{op}\otimes \mathcal{U}(A)$ as a vector space. Thus $\mathcal{Q}(A)$ has a PBW-type basis given by
 \[\{v_i\otimes v_j\# \overrightarrow{\alpha}\mid i, j\in S, \alpha = (i(1), i(2),\cdots, i(r))\in S^r,\
i(1) \le \cdots \le i(r), \ r\ge 0\}.\]
The identity in $\mathcal{Q}(A)$ is given by $1_A\otimes 1_{A^{op}}\# 1_{\mathcal{U}(A)}$, and the multiplication in $\mathcal{Q}(A)$ can be written down explicitly as
\begin{equation*}
(\ v_{i_1}\otimes v_{j_1}\#\overrightarrow{\alpha})(
v_{i_2}\otimes v_{j_2}\#\overrightarrow{\beta}\ )=
\sum_{\alpha=\alpha_1\sqcup\alpha_2\sqcup\alpha_3}\
\!\!\!\!\!\!\!\!\!\!(v_{i_1}\sbullet\overrightarrow{\alpha_1}(v_{i_2}))\otimes
(v_{j_1}\scirc \overrightarrow{\alpha_2}(v_{j_2}))
\#(\overrightarrow{\alpha_3}\overrightarrow{\beta}),
\end{equation*}
here $\sbullet$ and $\scirc$ denote the multiplications in $A$ and $A^{op}$ respectively.
\end{rem}

\begin{rem}\label{Poi-env-difference} Recall that for a usual Poisson algebra $A$,
another version of Poisson enveloping algebra, denoted by $\mathcal{P'}(A)$ as distinguished from $\mathcal{P}(A)$, has
already been developed; see \cite[Definition 3]{oh}, for example.
This notion differs from ours, the reason is that the Poisson modules concerned are
different, just as mentioned in Remark \ref{poi-difference}. In fact,
$\mathcal{P'}(A)= \mathcal{P}(A)/\langle (a\otimes 1_{A^{op}} - 1_A\otimes a) \# 1_{\mathcal{U}(A)}, a\in A \rangle$.
\end{rem}

\begin{rem}\label{rem-embedding}
We have natural morphisms of associative algebras:
\begin{align*}
& i\colon A\hookrightarrow \mathcal{Q}(A),\ a\mapsto a\otimes1_{A^{op}}\#
1_{\mathcal{U}(A)}; & i'\colon A\hookrightarrow \mathcal{Q}(A) \twoheadrightarrow \mathcal{P}(A);\\
& k\colon A^{op}\hookrightarrow \mathcal{Q}(A),\ a\mapsto 1_{A}\otimes a\#
1_{\mathcal{U}(A)};
& k'\colon A^{op}\hookrightarrow \mathcal{Q}(A) \twoheadrightarrow \mathcal{P}(A);
\end{align*}
and morphisms of Lie algebras:
\begin{align*} &j\colon A\hookrightarrow
\mathcal{Q}(A),\ a\mapsto 1_A\otimes 1_{A^{op}}\# a; &j'\colon A\hookrightarrow\mathcal{Q}(A)\twoheadrightarrow
\mathcal{P}(A).
\end{align*}
By construction of smash product we know that $i, j,  k$ are all injective maps. Let $f\colon \mathcal{P}(A) \to A$ be the linear map given by $f(a\otimes b\# \overrightarrow{\alpha})= \epsilon(\overrightarrow{\alpha})ab$. Clearly $f\circ i'= f\circ k'= \Id_A$, which implies that $i'$ and $k'$ are both injective maps. Note that $j'$ is not an injective map, for $0\ne 1_A\in \operatorname{Ker}(j')$.
\end{rem}

Clearly $i(A), j(A)$ and $k(A)$ generate $\mathcal{Q}(A)$ and hence $\mathcal{P}(A)$. Moreover, the generating relations can be written down explicitly.

\begin{prop} \label{gen-rel}
Let  $A$ be an NCPA. Then  $\mathcal{Q}(A)$ is generated by the generators $i(a), j(a), k(a), a\in A$ subject to relations:
\begin{align*}
&i(a)i(b)=i(a\sbullet b),\ k(a)k(b)=k(a\scirc b),\ i(a)k(b)=k(b)i(a); \\
&j(a)j(b)-j(b)j(a) = j(\{a,b\});\\
&j(a)i(b)=i(b)j(a)+i(\{a,b\}),\ j(a)k(b)= k(b)j(a) + k(\{a,b\}),\ \forall a,b\in A.
\end{align*}
\end{prop}

The proof is given by routine check and we omit it here.

\begin{exm}
In case $A$ is an NCPA with trivial Poisson bracket, it is plain to see that
   $\mathcal{Q}(A)= A\otimes A^{op}\otimes S(A)$ as an associative algebra, here
 $S(A)$ is the polynomial ring of $A$.

\end{exm}

\begin{exm} Let $A= \mathbb{K} \times \mathbb{K}$ as an associative
algebra. Then there exists a unique Lie bracket, say the commutative one, making $A$ an NCPA.
We claim that in this case, $\mathcal{P}(A)\cong \mathbb{K}\times \mathbb{K}\times \mathbb{K}[x]\times \mathbb{K}[x]$.

In fact, $1_A= e_1 + e_2$ is an orthogonal decomposition of identity, here $e_1= (1,0)$
and $e_2= (0,1)$.  Set $e_{st} = i'(e_s)k'(e_t)$ for all $1\le s,t\le 2$,
$\alpha_{12} = i'(e_1)j'(e_1)k'(e_2)$ and $\alpha_{21} = i'(e_2)j'(e_2)k'(e_1)$.
For any subset $S\subseteq \mathcal{P}(A)$, we denote by $\mathbb{K}\lceil S\rceil$ the subalgebra of $\mathcal{P}(A)$
generated by $S$, note that we do not require that $1_{\mathcal{P}(A)}\in \mathbb{K}\lceil S\rceil$ here.

Consider subalgebras $A_{11}= \mathbb{K}\lceil e_{11}\rceil $, $A_{22}= \mathbb{K}\lceil e_{22}\rceil $,
$A_{12}= \mathbb{K}\lceil e_{12}, \alpha_{12}\rceil $ and $A_{21}= \mathbb{K}\lceil e_{21}, \alpha_{21}\rceil $.
Clearly $A_{11} \cong \mathbb{K}\cong A_{22}$ and $A_{12} \cong \mathbb{K}[x]\cong A_{21}$ as algebras. Moreover,
$\mathcal{P}(A) = A_{11} \times  A_{22}\times A_{12}\times A_{21} $ and the assertion follows.
\end{exm}

\begin{exm}Let $A=\mathbb{K}\langle x_1, x_2,\cdots, x_n\rangle/J^2$  be a 2-truncated algebra, where
$J= \langle x_1, x_2, \cdots, x_n \rangle$ is the graded Jacobson radical. Let $\{-,-\}$ be a
Lie bracket on $A$. Then $(A, \sbullet, \{-,-\})$ is an NCPA if and only if $1_A$ is a center
element in the Lie algebra $A$.

 Set $\alpha_i= i'(x_i)$, $\beta_i = k'(x_i)$ and $\gamma_i = j'(x_i)$ for any $1\le i\le n$, here $i', j', k'$ be  given as in Remark \ref{rem-embedding}. Then $\mathcal{P}(A)$ has a basis consisting of
\begin{gather*} \alpha_i, \beta_j, \alpha_i\beta_j, 1\le i, j\le n,\\
 \gamma_1^{a_1}\gamma_2^{a_2}\cdots \gamma_n^{a_n}, a_i\ge 0, i=1,2,\cdots, n,\\
 \alpha_i \gamma_1^{a_1}\gamma_2^{a_2}\cdots \gamma_n^{a_n}, a_i\ge 0, i=1,2,\cdots, n, \sum_{i=1}^{n}a_i >0.\end{gather*}
\end{exm}

\begin{rem}To calculate the Poisson enveloping algebras of an NCPA could be very difficult in general.
An open question is to ask whether there exist PBW-like bases for Poisson enveloping
algebras.
\end{rem}

\section{Poisson modules are modules}

We are in a position to exhibit the equivalence between the category of (quasi)-Poisson modules over an
NCPA and the category of modules over its (quasi)-Poisson enveloping algebras.

Let $A$ be an NCPA and $M\in \QB(A)$ a quasi-Poisson module. To avoid
confusion we fix some notations. The left, right and Lie action of $A$ on $M$ are denoted by $\cdotl$, $\cdotr$ and $\{-,-\}_*$ respectively. The induced action of $\mathcal{U}(A)$ on $M$ is denoted by $x(m)$ for any $x\in \mathcal{U}(A)$ and $m\in M$, to be precise, $\overrightarrow{\alpha}(m) = \{v_{i(1)},\{v_{i(2)},\{,\cdots,\{v_{i(r)}, a \}_*\cdots\}_*\}_*\}_*$ for any $\alpha=(i_1,\cdots,i_r)\in S^r$ and $m\in M$. By the similar argument used in the proof of Proposition \ref{prop-A-is-mod-alg}, we show that $x (a\cdotl m) = \sum x_1(a)\cdotl x_2(m)$ and $x(m\cdot b)=\sum x_1(m)\cdotr x_2(b)$ for any $x\in \mathcal{U}(A)$, $a, b\in A$ and $m\in M$, again we use the Sweedler's notation here. The action of $\mathcal{Q}(A)$
is always denoted by $\cdot$. The multiplications in $A$ and $A^{op}$ are denoted by $\sbullet$ and $\scirc$, respectively.

Now we define $F\colon \QB(A)\longrightarrow \Mod(\mathcal{Q}(A))$ as follows. Let $M\in \QB(A))$. We define an action of $\mathcal{Q}(A)$ on the underlying space $M$ by setting
\[ (v_i\otimes
v_j\#\overrightarrow{\alpha})\cdot m \triangleq v_i\cdotl(\overrightarrow{\alpha}(m))\cdotr v_j\]
for any $i, j\in S$, $\alpha\in S^r$ and $m\in M$. We claim that this action makes $M$ a $\mathcal{Q}(A)$-module. For this it suffices to show that the associativity
\begin{equation}
(v_{i_1}\otimes v_{j_1} \#\overrightarrow{\alpha}) ((v_{i_2}\otimes
v_{j_2}\#\overrightarrow{\beta})(m))= ((v_{i_1}\otimes v_{j_1}
\#\overrightarrow{\alpha}) (v_{i_2}\otimes
v_{j_2}\#\overrightarrow{\beta}))(m)
\end{equation}
holds for any $m\in M$, $i_1, i_2, j_1, j_2\in S$, $\alpha$ and $\beta$. One uses Proposition \ref{gen-rel} to give a proof; or one can give a direct argument as follows. In fact,
 the left hand side
 \begin{align*}LHS=&\ (v_{i_1}\otimes
v_{j_1}\#\overrightarrow{\alpha})
(v_{i_2}\cdotl(\overrightarrow{\beta}(m))\cdotr v_{j_2})
\\=&\ v_{i_1}\cdotl\overrightarrow{\alpha}(v_{i_2}\cdotl\overrightarrow{\beta}(m)\cdotr
v_{j_2})\cdotr v_{j_1}
\\ =&\ \sum_{\alpha=\alpha_1\sqcup\alpha_2\sqcup\alpha_3}
v_{i_1}\cdotl(\overrightarrow{\alpha_1}(v_{i_2})\cdotl
(\overrightarrow{\alpha_3}\overrightarrow{\beta})(m)\cdotr\overrightarrow{\alpha_2}
(v_{j_2}))\cdotr v_{j_1}
\\ =&\ \sum_{\alpha=\alpha_1\sqcup\alpha_2\sqcup\alpha_3}
(v_{i_1}\sbullet\overrightarrow{\alpha_1}(v_{i_2}))\cdotl
(\overrightarrow{\alpha_3}\overrightarrow{\beta})(m)\cdotr(v_{j_1}\scirc \overrightarrow{\alpha_2}
(v_{j_2}))
\\=&\ \sum_{\alpha=\alpha_1\sqcup\alpha_2\sqcup\alpha_3}
(v_{i_1}\sbullet\overrightarrow{\alpha_1}(v_{i_2})\otimes
v_{j_1}\scirc \overrightarrow{\alpha_2} (v_{j_2})\#
(\overrightarrow{\alpha_3}\overrightarrow{\beta}))(m)
\\=&\ RHS,
\end{align*}
 the right hand side. We denote by $FM$ the $\mathcal{U}(A)$-module with underlying space $M$ and the
 above action.

Next we define $G\colon\Mod(\mathcal{Q}(A))\longrightarrow \QB(A)$. Given $M\in
\Mod(\mathcal{Q}(A))$, we may endow the underlying space $M$
 a quasi-Poisson $A$-module structure by setting
\begin{align}
&a\cdotl m\triangleq i(a)\cdot m,\ m\cdotr a\triangleq
k(a)\cdot m, \ \forall a\in A,\ m\in M;\\&\{a, m\}_*\triangleq
j(a)\cdot m,\ \forall a\in A,\ m\in M.
\end{align}
Clearly $M$ is an $A$-bimodule as well as a Lie module with the above actions, since $i, k$ are
homomorphisms of associative algebras and $j$ is a homomorphism of Lie algebras.
We claim that $M$ is a quasi-Poisson module with the above actions. For this it suffices to show that
(\ref{pm1}) and (\ref{pm2}), which can be translated to
\begin{align*}
&j(a)\cdot(i(b)\cdot m) = i(\{a,b\})\cdot
m + i(b)\cdot(j(a)\cdot m),\\&j(a)\cdot(k(b)\cdot m)=
k(\{a,b\})\cdot m + k(b)\cdot(j(a)\cdot m),
\end{align*}
hold for any $a,b\in A$ and $m\in M$.  This is obvious since in
$\mathcal{Q}(A)$,
\begin{align*}
&j(a)i(b)=i(\{a,b\}) + i(b)j(a),
\\&j(a)k(b)=k(\{a,b\}) + k(b)j(a).
\end{align*}
We denote by $GM$ the quasi-Poisson $A$-module with the underlying space $M$ and the above actions.

Recall that there is a forgetful functor, denoted by $\mathbb{F}_{\QB(A)}$ from $\QB(A)$ to $\Mod({\mathbb{K}})$, the category of $\mathbb{K}$-vector spaces. Similarly we have forgetful functors $\mathbb{F}_{\Mod(\mathcal{Q}(A))}$, $\mathbb{F}_{\PB(A)}$ and $\mathbb{F}_{\Mod(\mathcal{P}(A))}$.

\begin{lem} Let $F\colon\QB(A)\rightarrow \Mod(\mathcal{Q}(A))$ and
$G\colon\Mod(\mathcal{Q}(A))\rightarrow \QB(A)$ be given as above. Then $F$ and $G$
give rise to functors respecting the forgetful functors to $\Mod({\mathbb{K}})$, i.e. $\mathbb{F}_{\QB(A)}=\mathbb{F}_{\Mod(\mathcal{Q}(A))} \circ F$ and $\mathbb{F}_{\Mod(\mathcal{Q}(A))} = \mathbb{F}_{\QB(A)} \circ G$.
\end{lem}

\begin{proof} We consider $F$ first. Note that for each $M \in \QB(A)$, we have defined an object
$FM\in \Mod(\mathcal{Q}(A))$ with $ FM$ sharing the same underlying space with $M$. To extend $F$ to
a functor, we need to define a map
\[F\colon \Hom_{\QB(A)}(M, N) \longrightarrow \Hom_{\mathcal{Q}(A)}(FM, FN)\]
for each pair of $M, N\in \QB(A)$, such that certain compatible conditions are satisfied. Now the requirement
to respect the forgetful functor means that for each $M\in \QB(A)$, $M$ and $FM$ have
the same underlying vector space, and for each $f\in \Hom_{\QB(A)}(M, N)$, $Ff = f$ as
morphisms of vector spaces.

To complete the proof we need only to show that if $f\colon M\longrightarrow N$ is a
morphism of quasi-Poisson modules, then $f$ is a morphism of $\mathcal{Q}(A)$-modules.
 This is easy to show. In fact, we have $f (v_i\cdotl \overrightarrow{\alpha}(m)\cdotr v_j) = v_i\cdotl \overrightarrow{\alpha}(f(m))\cdotr v_j$ and hence
\[f((v_i\otimes v_j\#\overrightarrow{\alpha})\cdot m)
= (v_i\otimes v_j\#\overrightarrow{\alpha})\cdot f(m)\]
for all $m\in M, i, j\in S$ and $\alpha$.

Similar argument works for $G$ and the lemma follows.
\end{proof}

Recall that the Poisson enveloping algebra $\mathcal{P}(A)$ of $A$ is by definition the
quotient algebra $\mathcal{Q}(A)/ J$,  where $J$ is the ideal of $\mathcal{Q}(A)$ generated by
\[\{1_{A}\otimes1_{A^{op}}\# (a\sbullet b)-a\otimes1_{A^{op}}\# b-1_{A}\otimes b\# a|\
 a,b\in A\}.\]
Thus $\Mod(\mathcal{P}(A))$ is a full subcategory of $\Mod(\mathcal{Q}(A))$, and a $\mathcal{Q}(A)$-module
$M$ is a $\mathcal{P}(A)$-module if and only if $J$ annihilates $M$.

 Now consider the restriction of the functor $F$ to $\PB(A)$, the full subcategory consisting
 of Poisson modules. Let $M\in \PB(A)$. One checks easily that in $FM$,
\begin{align*}  &(1_{A}\otimes1_{A^{op}}\# (a\sbullet b)-a\otimes1_{A^{op}}
 \# b-1_{A}\otimes b\# a)\cdot m\\
=& \{a\sbullet b, m\}_* - a\cdotl \{b, m\}_* - \{a, m\}_*\cdotr b = 0.
\end{align*}
The first equality follows from the construction of $FM$ and the latter one from the definition of
Poisson modules. Thus $J$ annihilates $FM$ and hence $FM$ is a $\mathcal{P}(A)$-module.
Therefore $F|_{\PB(A)}$ gives a functor $ F'\colon\PB(A)\longrightarrow \Mod(\mathcal{P}(A))$.

Conversely, the above equalities also implies that for each
 $\mathcal{P}(A)$-module $M$, $GM$ satisfies (\ref{pm4}), and
hence is a Poisson $A$-module. Thus $G|_{\Mod(\mathcal{P}(A))}$ gives a
functor $G'\colon \Mod(\mathcal{P}(A))\longrightarrow\PB(A)$.

Now our main result is stated as follows.

\begin{thm}
 Let $A$ be an NCPA. Then there are isomorphisms of categories $\QB(A) \cong \Mod(\mathcal{Q}(A)) $
and $\PB(A) \cong \Mod(\mathcal{P}(A)) $. In fact, we have
\begin{align}
&F\circ G = \Id_{\Mod(\mathcal{Q}(A))},\ G\circ F =\Id_{\QB(A)},
\\&F'\circ G'=\Id_{\Mod(\mathcal{P}(A))},\
G'\circ F'=\Id_{\PB(A)},
\end{align}
here $F, G, F', G'$ are functors given as above.
\end{thm}
\begin{proof} First we show that $G \circ F(M)=M$. Since
$G \circ F(M)$ is defined to have the same underlying vector
space as $M$, it is only left to prove that they have the same actions of $A$. By definitions of $F$ and $G$, we have
\begin{align*}
&a\cdot_{G\scirc F(M)}(m)=i(a)\cdot_{F(M)}(m)=a\cdot_{M}(m),
\\&m\cdot_{G\scirc
F(M)}a=k(a)\cdot_{F(M)}(m)=m\cdot_{M}a,\\&\{a,m\}_{*G\scirc
F(M)}=j(a)\cdot_{F(M)}(m)=\{a,m\}_{*M}
\end{align*}
for all $a\in A$ and $m\in M$, where $\cdot_{G\scirc F(M)}$, $\{-,-\}_{*G\scirc F(M)}$,
$\cdot_{M}$ and $\{-,-\}_{*M}$ denote the actions of $A$ on the modules
$G\circ F(M)$ and $M$ respectively. Thus we have shown that $G\circ F(M) = M$ as quasi-Poisson $A$-modules.

Next we show that $F\circ G(M)= M$ for all $M\in
\Mod(\mathcal{Q}(A))$, again it suffices to show that \[(v_i\otimes
v_j\#\overrightarrow{\alpha})\cdot_{F\scirc G(M)} m=(v_i\otimes
v_j\#\overrightarrow{\alpha})\cdot_M m\] holds for any $v_i\otimes
v_j\#\overrightarrow{\alpha}\in \mathcal{Q}(A)$ and $m\in M$. This is true since by definition
\begin{align*}
&(v_i\otimes v_j\#\overrightarrow{\alpha})\cdot_{F\scirc G(M)} m\\
=& v_i \cdot_{G(M)}
(\overrightarrow{\alpha}_{G(M)}(m))\cdot_{G^1(M)}v_j\\
=& ((1_{A}\otimes v_j\# 1_{\mathcal{U}(A)})
(v_i\otimes1_{A^{op}}\# 1_{\mathcal{U}(A)})(1_{A}\otimes1_{A^{op}}\#\overrightarrow{\alpha}))\cdot_M m\\
=&(v_i\otimes v_j\#\overrightarrow{\alpha})\cdot_M m.
\end{align*}

Recall that $F$ and $G$ are both functors respecting the forgetful functor, which means
 $G\circ F(f) =f$ for each morphism $f$ in the category $\QB(A)$. It follows that
 $G\circ F =\Id_{\QB(A)}$ as functors.

Similarly, we obtain that $F\circ G = \Id_{\Mod(\mathcal{Q}(A))} $. Since $F'$ is a
restriction of $F$ and $G'$ is a restriction of $G$, we know that $F'\circ G'$ is a restriction
of $F\circ G$, and hence we deduce easily that $F'\circ G'=\Id_{\Mod(\mathcal{P}(A))}$.
Similarly we have $G'\circ F'=\Id_{\PB(A)}$. The proof is completed.
\end{proof}

\begin{rem}
An easy consequence of the theorem is that the categories $\QB(A)$ and $\PB(A)$ are
both abelian categories with enough projectives and injectives, which answers a
question raised in \cite[Section 2]{fgv}. The quasi-Poisson and Poisson enveloping
algebras also enable us to apply the ring theoretic methods to the study of NCPA,
especially to the study of cohomology theory of NCPA.
\end{rem}

\begin{rem} An NCPA $A$ itself can be viewed as a
$\mathcal{P}(A)$-module via the functor $F'$ and a Poisson ideal of
$A$ is just a $\mathcal{P}(A)$-submodule of $A$. Now $A$ is Poisson
simple if and only if $A$ is simple as a $\mathcal{P}(A)$-module.
\end{rem}

\begin{exm} Let $A$ be an associative algebra and consider the standard NCPA of $A$. It is direct to check that any $A$-bimodule $M$ is a Poisson $A$-module under the Lie action given by $\{a, m\}_* = a\cdot m-m\cdot a$. This gives a functor from the category $\Mod(A\otimes A^{op})$ to the category $\PB(A)$, and hence to $\Mod(\mathcal{P}(A))$. In fact, this functor is induced from an isomorphism of associative algebras $\psi\colon A\otimes A^{op}\cong \mathcal{P}(A)/I$, where $I= \langle j(a) - i(a) + k(a), a\in A\rangle$ is a two-sided ideal of $\mathcal{P}(A)$.

The isomorphism $\psi$ is given by the composition map $ A\otimes A^{op} \xrightarrow{f} \mathcal{P}(A) \xrightarrow{\pi} \mathcal{P}(A)/I$, here $f$ is defined by $f(a\otimes b) = i'(a)k'(b)$ for any $a,b\in A$ and $\pi$ is the canonical map. Obviously $\psi$ is an epimorphism, hence to show that $\psi$ is an isomorphism it suffices to show that it is also injective. For this, one uses the fact that the $A$-bimodule $A\otimes A^{op}$ is a Poisson module under the Lie action given by $\{c, a\otimes b\}_*= ca\otimes b - a\otimes bc$ for all $a, b, c\in A$. Thus $A\otimes A^{op}$ is a $\mathcal{P}(A)$-module. Moreover, the ideal $I$ annihilates $A\otimes A^{op}$ and hence $A\otimes A^{op}$ is a $\mathcal{P}(A)/I$-module. By definition of the functor $F$, $\psi(x)\cdot (1_A\otimes 1_{A^{op}}) = x $ for any $x\in A\otimes A^{op}$. It follows that $\psi(x)\ne 0$ for any $0\ne x\in A\otimes A^{op}$, which implies that $\psi$ is injective.
\end{exm}

\noindent{\bf Acknowledgements:} This work is supported by National Natural Science Foundation of China (No. 10971206). The authors would like to thank Dr. Yan-Hong Bao and Dr. Xiao-Wu Chen for helpful suggestions.

\bibliography{}

\begin{thebibliography}{99}
\bibitem{bh} Y.H. Bao and Y. Ye, On quasi-Poisson cohomology, in preparation.

\bibitem{ceg} W. Crawley-Boevey, P. Etingof and V. Ginzburg, Noncommutative geometry and quiver algebras. {\it Adv. Math.} 209
(2007) 274-336.

\bibitem{cp} J.M. Casas and T. Pirashvili, Algebras with bracket. {\it Manuscripta Math.} 119 (2006)  1-15.

\bibitem{fa} D. Farkas, Modules for Poisson Algebras. {\it Comm. in Algebra}  28 (7) (2000) 3293-3306.

\bibitem{fgv} M. Flato, M. Gerstenhaber and A.A. Voronov,
 Cohomology and deformation of Leibniz pairs.
{\it Lett. Math. Phys.}  34 (1995), 77-90.

\bibitem{fl} D. Farkas and G. Letzter, Ring Theory from
Symplectic Geometry. {\it Journal of Pure and Applied Algebra} 225(1998)  255-290.

\bibitem{gk} V. Ginzburg and M. Kapranov, Koszul duality for operads. {\it Duke Math. J.}  76 (1) (1994)  203-272.

\bibitem{ko} M. Kontsevich, Formal (non)commutative symplectic
geometry. {\it The Gelfand Math. Seminars}, 1990-1992, Birkh\"auser,
Boston, MA, 173-187.

\bibitem{ku1} F. Kubo, Finite-dimensional Simple Leibniz Pairs and
Simple Poisson Modules. {\it Lett. Math.Phys.} 43 (1) (1998) 21-29.

\bibitem{lo} F. Loose, Symplectic algebras and Poisson algebras.
 {\it Commun. Algebras} 21 (7) (1993) 2395-2416.

\bibitem{oh}Sei-Qwon Oh, Poisson Enveloping Algebras.
 {\it Comm. Algebra}  27 (5) (1999)  2181-2186.

\bibitem{rvw}N. Reshetikhin, A. A. Voronov and A. Weinstein,
Semiquantum geometry. {\it J. Math. Sci.}  82 (1) (1996)  3255--3267.

\bibitem{sw} M.E. Sweedler, Hopf algebras, W. A. Benjamin, Inc., New York, 1969.

\bibitem{va} I. Vaisman, Lectures on the geometry of Poisson
manifolds, Progress in Math., Vol. 118, Birkh\"{a}user, Basel, 1994.

\bibitem{vdb} M. Van den Bergh, Double Poisson Algebras.
{\it Trans. Amer. Math. Soc.}  360 (11) (2008)  5711--5769.

\bibitem{xu} P. Xu, Noncommutative Poisson algebras.
{\it Amer. J. Math.}  116 (1) (1994)  101-125.

\bibitem{yyz}Y. Yao, Y.Ye and P.Zhang, Quiver Poisson
Algebras. {\it J. Algebra} 312 (2) (2007), 570--589.

\end{thebibliography}

\end{document}